\documentclass[11pt,reqno]{amsart}

\usepackage{hyperref}
\usepackage{graphicx}
\usepackage{color}
\usepackage{amsfonts,amssymb}
\usepackage{bbm} 

\usepackage{mathrsfs}

\usepackage{pdfsync}

\usepackage{a4wide}

\newtheorem{neu}{}
\newtheorem{Cor}[neu]{Corollary}
\newtheorem*{Cor*}{Corollary}
\newtheorem{Thm}[neu]{Theorem}
\newtheorem*{Thm*}{Theorem}
\newtheorem*{Observation*}{Observation}

\newtheorem*{Prop*}{Proposition}
\newtheorem{Lemma}[neu]{Lemma}
\theoremstyle{definition}\newtheorem*{Rmk*}{Remark}
\newtheorem{Rmk}[neu]{Remark}

\newtheorem*{Ex*}{Example}

\newtheorem*{Qu*}{Question}

\newcommand{\R}{\mathbb{R}}
\newcommand{\C}{\mathbb{C}}

\newcommand{\om}{\omega}

\newcommand{\A}{\mathscr{A}}

\renewcommand{\L}{\mathscr{L}}

\newcommand{\beq}{\begin{equation}}
\newcommand{\beqn}{\begin{equation}\nonumber}
\newcommand{\eeq}{\end{equation}}

\newcommand{\bea}{\begin{equation}\begin{aligned}}
\newcommand{\bean}{\begin{equation}\begin{aligned}\nonumber}
\newcommand{\eea}{\end{aligned}\end{equation}}

\definecolor{Urs}{rgb}{0,.7,0}
\definecolor{Peter}{rgb}{0,0,1}
\definecolor{red}{rgb}{1,0,0}

\newcommand{\Sym}{\mathrm{Sym}}
\newcommand{\Gl}{\mathrm{Gl}}
\newcommand{\Sp}{\mathrm{Sp}}
\newcommand{\1}{\mathbbm{1}}

\begin{document}
\title[The space of linear anti-symplectic involutions is a homogenous space]{The space of linear anti-symplectic involutions\\ is a homogenous space}
\author{Peter Albers}
\author{Urs Frauenfelder}
\address{
    Peter Albers\\
    Mathematisches Institut\\
    Westf\"alische Wilhelms-Universit\"at M\"unster}
\email{peter.albers@wwu.de}
\address{
    Urs Frauenfelder\\
    Department of Mathematics and Research Institute of Mathematics\\
    Seoul National University}
\email{frauenf@snu.ac.kr}
\begin{abstract}
In this note we prove that the space of linear anti-symplectic involutions is the homogenous space $\mathrm{Gl}(n,\R)\backslash\Sp(n)$. This result is motivated by the study of symmetric periodic orbits in the restricted 3-body problem.
 \end{abstract}
\maketitle


\section{Introduction}

We denote by $(\R^{2n},\om)$ the standard symplectic vector space. Then
\beq
\Sp(n):=\{\psi\in\Gl(2n,\R)\mid \om(\psi v,\psi w)=\om(v,w),\;\forall v,w\in\R^{2n}\}
\eeq
is the linear symplectic group. We denote by $\A(n)$ the space of linear anti-symplectic involutions
\beq
\A(n):=\{S\in\Gl(2n,\R)\mid S^2=\1\quad\text{and}\quad \om(S v,S w)=-\om(v,w),\;\forall v,w\in\R^{2n}\}
\eeq
and by
\beq
R:=
\begin{pmatrix}
\1 & 0\\
0 & -\1 
\end{pmatrix}
\in\A(n)
\eeq
the standard anti-symplectic involution written in matrix form with respect to the Lagrangian splitting $\R^{2n}\cong\R^n\oplus\R^n\cong T^*\R^n$. Finally, we abbreviate
\beq
\Sp^R(n):=\{\psi\in\Sp(n)\mid \psi=R\psi^{-1}R\}\;.
\eeq
The group $\mathrm{Gl}(n,\R)$ is identified with the subgroup 
\beq\label{eqn:Gl(n,R)}
\left\{
\begin{pmatrix}
A & 0\\
0 & (A^T)^{-1} 
\end{pmatrix}\mid A\in\mathrm{Gl}(n,R)
\right\}
\eeq
of $\Sp(n)$. In this note we prove the following Theorems.

\begin{Thm}\label{thm:theorem1}
$\Sp^R(n)$, $\A(n)$, and the homogeneous space $\mathrm{Gl}(n,\R)\backslash\Sp(n)$ are diffeomorphic.
\end{Thm}

\begin{Thm}\label{thm:theorem2}
Every element in $\Sp(n)$ is conjugate in $\Sp(n)$ to an element in $\Sp^R(n)$, i.e.~for all $\phi\in\Sp(n)$ there exists $\psi\in\Sp(n)$ with $\psi\phi\psi^{-1}\in\Sp^R(n)$.
\end{Thm}

\begin{Rmk}
The principal interest in the space $\Sp^R(n)$ comes from studying symmetric periodic orbits of Hamiltonian dynamical systems invariant under anti-symplectic involution. Indeed, the linearized Poincar\'e return map takes values in $\Sp^R(n)$. In the restricted 3-body problem an antisymplectic involution plays a crucial role already in the work of Birkhoff \cite{Birkhoff_The_restricted_problem_of_three_bodies}. In fact, the term \textit{symmetric periodic orbit} originates from Birkhoff's work. For more details we refer the reader to \cite{Frauenfelder_vKoert_Hormadner_index_of_symmetric_periodic_orbits}. 

As in symplectic field theory it is important to understand the dichotomy of good and bad periodic orbits, cp.~\cite{Eliashberg_Givental_Hofer_SFT}. This property only depends on the conjugacy class of the linearized Poincar\'e return map. In particular, by Theorem \ref{thm:theorem2} the property of being \textit{symmetric} does not pose any obstructions on the conjugacy class of the linearized Poincar\'e return map.
\end{Rmk}

\begin{Rmk}
The homogeneous space structure is not unique.  At the end of this article we give the Lagrangian Grassmannian a second homogenous structure, see Corollary \ref{cor}.
\end{Rmk}

\section{Proof of Theorems \ref{thm:theorem1} and \ref{thm:theorem2}}

As preparation we need 
\begin{Lemma}\label{lem:Gl(n,R)_=_commute_with_R}
$\mathrm{Gl}(n,\R)$ equals
\beq
\{\psi\in\Sp(n)\mid R\psi=\psi R\}\;.
\eeq
\end{Lemma}

\begin{proof}
That $\mathrm{Gl}(n,\R)$ is contained in the given set follows from equation \eqref{eqn:Gl(n,R)}. For the opposite inclusion we recall that a matrix
\beq
\psi:=\begin{pmatrix}
A & B\\
C & D 
\end{pmatrix}
\eeq
is in $\Sp(n)$ if and only if 
\beq
AD^T-C^TB=\1,\; A^TC=C^TA,\quad\text{and}\quad B^TD=D^TB\;,
\eeq
see \cite[Exercise 1.13]{McDuff_Salamon_introduction_symplectic_topology}. Thus, the equality
\beq
R\psi R=
\begin{pmatrix}
A & - B \\
-C & D 
\end{pmatrix}
\stackrel{!}{=}\psi
\eeq
implies that $B=C=0$ and since $\psi$ is symplectic we conclude from $AD^T-C^TB=\1$ that $D^T=A^{-1}$.
\end{proof}

The following Lemma is well-known. For the readers convenience we include a proof.
\begin{Lemma}\label{lem:existence_of_symplectic_basis_for_Lagr_splitting}
Let $\R^{2n}=L_1\oplus L_2$ be a Lagrangian splitting of $\R^{2n}$, i.e.~$L_1,L_2$ are Lagrangian subspaces, then there exists bases $v_1,\ldots,v_n$ of $L_1$ and $w_1,\ldots,w_n$ of $L_2$ such that
\beq
\om(v_i,w_j)=\delta_{ij},\quad i,j=1,\ldots,n
\eeq
that is, $v_1,\ldots,v_n,w_1\ldots,w_n$ is a symplectic basis of $\R^{2n}$.
\end{Lemma}

\begin{proof}
We denote by $\Pi_1:\R^{2n}\to\R^{2n}$ the projection to $L_1$ along $L_2$ and by $\Pi_2=\1-\Pi_1$ the projection to $L_2$ along $L_1$.
We prove the Lemma by induction. For $1\leq k\leq n$ we denote by A(k) the following assertion.

\textbf{A(k)}:\quad There exist linearly independent vectors $v_1,\ldots,v_k\in L_1$  and $w_1,\ldots,w_k\in L_2$ such that
\beq
\om(v_i,w_j)=\delta_{ij},\quad i,j=1,\ldots,k\;.
\eeq

\textbf{A(1)}:\quad We choose $v_1\in L_1\setminus\{0\}$. Then there exists $\hat{w}_1\in\R^{2n}$ with 
\beq
\om(v_1,\hat{w}_1)=1\;.
\eeq 
We set
\beq
w_1:=\Pi_2(\hat{w}_1)
\eeq
and compute using the fact that $L_1$ is Lagrangian that
\beq\label{eqn:remove_hat}
1=\om(v_1,\hat{w}_1)=\om(v_1,\underbrace{\Pi_1(\hat{w}_1)}_{\in L_1})+\om(v_1,\underbrace{\Pi_2(\hat{w}_1)}_{=w_1})=\om(v_1,w_1)\;.
\eeq
This verifies \textbf{A(1)}. Next we assume that $k\leq n-1$ and prove \textbf{A(k)} $\Rightarrow$ \textbf{A(k+1)}:\quad Since
\beq
\dim\left(\bigcap_{i=1}^k\ker \om(\cdot,w_i)\cap L_1\right)\geq 2n-k-n\geq 1
\eeq
there exists $v_{k+1}\in L_1\setminus\{0\}$ with
\beq
\om(v_{k+1},w_i)=0,\quad 1\leq i\leq k\;.
\eeq
To show that $v_1,\ldots,v_{k+1}$ are linearly independent we assume that
\beq\label{eqn:lin_indep_ansatz}
\sum_{i=1}^{k+1}a_iv_i=0
\eeq
and compute $j=1,\dots,k$ using the induction hypothesis
\bea
0=\om\bigg(\sum_{i=1}^{k+1}a_iv_i,w_j\bigg)=\sum_{i=1}^{k+1}a_i\underbrace{\om(v_i,w_j)}_{\delta_{ij}}=a_j\;.
\eea
Therefore, equation \eqref{eqn:lin_indep_ansatz} is reduced to $a_{j+1}v_{j+1}=0$ and we conclude $a_1=\ldots=a_{j+1}=0$, proving linear independence. Since $\om$ is non-degenerate there exists $\hat{w}_{k+1}$ with
\beq
\om(v_{k+1},\hat{w}_{k+1})=1\;.
\eeq
We set
\beq
w_{k+1}:=\Pi_2(\hat{w}_{k+1})
\eeq
and conclude as in \eqref{eqn:remove_hat} that $\om(v_{k+1},w_{k+1})=1$. Proving that $w_1,\ldots,w_{k+1}$ are linearly independent is done exactly the same way as for $v_1,\ldots,v_{k+1}$. This finishes the proof.
\end{proof}

Now we prove Theorem \ref{thm:theorem1}.
\begin{proof}[Proof of Theorem \ref{thm:theorem1}]
We consider the map
\bea
\Sp^R(n)&\to\A(n)\\
\psi&\mapsto R\psi
\eea
which is well-defined since $R\psi$ is anti-symplectic and
\beq
(R\psi)^2=R\psi R\psi=\psi^{-1}\psi=\1\;.
\eeq
Moreover, the map is a diffeomorphism with inverse
\bea
\A(n)&\to\Sp^R(n)\\
S&\mapsto RS
\eea
which again is well-defined since $RS$ is symplectic and 
\beq
R(RS)^{-1}R=RSRR=RS\;.
\eeq
Thus, $\Sp^R(n)\cong\A(n)$. It remains to prove $\A(n)\cong\mathrm{Gl}(n,\R)\backslash\Sp(n)$. By Lemma \ref{lem:Gl(n,R)_=_commute_with_R} the conjugation map
\bea
\Sp(n)&\to\A(n)\\
\psi&\mapsto \psi^{-1}R\psi
\eea
descends to a map
\beq
\mathfrak{C}:\mathrm{Gl}(n,\R)\backslash\Sp(n)\to\A(n)\;.
\eeq
In order to check that this is an isomorphism we observe that any involution $S$ is diagonalizable with eigenvalues $\pm1$. Indeed, if we set
\beq
V_{\pm1}:=\{v\in\R^{2n}\mid Sv=\pm v\}
\eeq
then
\bea\label{eqn:R2n=V_1+V_-1}
\R^{2n}&\cong V_1\oplus V_{-1}\\
v&\mapsto \tfrac12 (v+Sv)+\tfrac12(v-Sv)\;.
\eea
$S$ being anti-symplectic implies that the linear spaces $V_{\pm1}$ are isotropic since
\beq
\om(v,w)=-\om(Sv,Sw)=-\om(v,w)
\eeq
if $v,w\in V_1$ or $v,w\in V_{-1}$. Thus, by \eqref{eqn:R2n=V_1+V_-1}, $V_{\pm1}$ are Lagrangian. Now we choose according to Lemma \ref{lem:existence_of_symplectic_basis_for_Lagr_splitting} a symplectic basis $v_1,\ldots,v_n,w_1,\ldots,w_n$ of $\R^{2n}$ with $v_1,\ldots,v_n$ being a basis of $L_1$ and  $w_1,\ldots,w_n$ being a basis of $L_2$. We set
\beq
\psi^{-1}:=
\begin{pmatrix}
v_1,\dots,v_n,w_1,\dots,w_n 
\end{pmatrix}
\in\Sp(n)\;.
\eeq
With this definition it follows that 
\beq
\psi^{-1}R\psi=S
\eeq
and therefore the map $\mathfrak{C}$ is surjective. To check injectivity we note that
\beq
\psi_1^{-1}R\psi_1=\psi_2^{-1}R\psi_2\quad\Longleftrightarrow\quad \psi_2\psi_1^{-1}R\psi_1\psi_2^{-1}=R\;.
\eeq
Thus by Lemma \ref{lem:Gl(n,R)_=_commute_with_R} the latter implies that $\psi_1\psi_2^{-1}\in\mathrm{Gl}(n,R)$.
\end{proof}

\begin{proof}[Proof of Theorem 2]
According to Wonenburger \cite[Theorem 2]{Wonenburger_Transformations_which_are_products_of_two_involutions} we can write any $\phi\in\Sp(n)$ as a product of two linear anti-symplectic involutions:
\beq
\phi=TS,\qquad S,T\in\A(n)\;.
\eeq
Using Lemma \ref{lem:existence_of_symplectic_basis_for_Lagr_splitting} as in the proof of Theorem \ref{thm:theorem1} we can find a symplectic matrix $\psi\in\Sp(n)$ s.t.
\beq
\psi^{-1}R\psi=S\;.
\eeq
If we set
\beq
\widetilde{\phi}:=\psi\phi\psi^{-1}=\psi T\psi^{-1} \psi S\psi^{-1}=\underbrace{\psi T\psi^{-1}}_{\widetilde{T}\in\A(n)} R=\widetilde{T}R
\eeq
then we conclude
\beq
R\widetilde{\phi}R=R\widetilde{T}RR=R\widetilde{T}=\widetilde{\phi}^{-1}
\eeq
i.e.
\beq
\widetilde{\phi}\in\Sp^R(n)\;.
\eeq
\end{proof}

\section{Concluding remarks}

\begin{Rmk}
It is well-know that the homogeneous space $U(n)/O(n)$ is diffeomorphic to $U(n)\cap\Sym(n)$, where $\Sym(n)$ is the space of symmetric matrices in $\Gl(n,\C)$:
\beq
U(n)/O(n)\cong\{\theta\in U(n)\mid \theta=\theta^T\}\;.
\eeq
This can be seen similarly as in the proof of Theorem \ref{thm:theorem1}. The same maps restricted to the space of \textit{orthogonal} anti-symplectic involutions give the above claimed identification.

We recall that $U(n)/O(n)$ is diffeomorphic to the Lagrangian Grassmannian $\L(n)$, that is the space of all linear Lagrangian subspaces of $\R^{2n}$, see \cite[Lemma 2.31]{McDuff_Salamon_introduction_symplectic_topology}
\end{Rmk}

We close with the following unexpected observation, see Corollary \ref{cor}. We point out that the space $\A(n)$ can be identified with the space of ordered Lagrangian splittings of $\R^{2n}$, cp.~equation \eqref{eqn:R2n=V_1+V_-1}. Moreover, $\A(n)$ has a natural projection $\pi:\A(n)\to\L(n)$ to the Lagrangian Grassmannian $\L(n)$ given by $\pi(S):=Fix(S)$. The fiber $\pi^{-1}(L)$ consists of all Lagrangian splittings of the from $L\oplus\widetilde{L}$. \cite[Lemmas 2.30, 2.31]{McDuff_Salamon_introduction_symplectic_topology} imply that every such $\widetilde{L}$ is the graph over $L^\perp$ of a symmetric matrix. Thus, we can identify $\pi^{-1}(L)\cong\Sym(n)$. Similarly the tangent space $T_L\L(n)\cong\Sym(n)$. We proved
\begin{Lemma}
The space of linear anti-symplectic involutions is diffeomorphic to the tangent bundle of the Lagrangian Grassmannian
\beq
\A(n)\cong T\L(n)\;.
\eeq 
\end{Lemma}

\begin{Cor}\label{cor}
The tangent bundle of the Lagrangian Grassmannian can be given the structure of a homogenous space in two different ways.
\end{Cor}

\begin{proof}
One homogeneous structure can be obtained from the diffeomorphism $\A(n)\cong\mathrm{Gl}(n,\R)\backslash\Sp(n)$, see Theorem \ref{thm:theorem1}. The other from the Theorem in \cite{Brockett_Sussmann_Tangent_budles_of_homogeneous_spaces_are_homogeneous_spaces} based on semi-direct products and is diffeomorphic to $TU(n)/TO(n)$.
\end{proof}

\subsection*{Acknowledgements}
This article was written during a visit of the authors at the Forschungs\-institut f\"ur Mathematik (FIM), ETH Z\"urich. The authors thank the FIM for its stimulating working atmosphere. 

This material is supported by the SFB 878 -- Groups, Geometry and Actions (PA) and by the Basic Research fund 20100007669 funded by the Korean government basic (UF).

%
\bibliographystyle{amsalpha}
\bibliography{../../../../Bibtex/bibtex_paper_list}

\providecommand{\bysame}{\leavevmode\hbox to3em{\hrulefill}\thinspace}
\providecommand{\MR}{\relax\ifhmode\unskip\space\fi MR }
\providecommand{\MRhref}[2]{%
  \href{http://www.ams.org/mathscinet-getitem?mr=#1}{#2}
}
\providecommand{\href}[2]{#2}
\begin{thebibliography}{EGH00}

\bibitem[Bir15]{Birkhoff_The_restricted_problem_of_three_bodies}
G.~D. Birkhoff, \emph{{The restricted problem of three bodies.}}, Rend. Circ.
  Mat. Palermo \textbf{39} (1915), 265--334.

\bibitem[BS72]{Brockett_Sussmann_Tangent_budles_of_homogeneous_spaces_are_homogeneous_spaces}
R.~W. Brockett and H.~J. Sussmann, \emph{Tangent bundles of homogeneous spaces
  are homogeneous spaces}, Proc. Amer. Math. Soc. \textbf{35} (1972), 550--551.

\bibitem[EGH00]{Eliashberg_Givental_Hofer_SFT}
Y.~Eliashberg, A.~Givental, and H.~Hofer, \emph{Introduction to symplectic
  field theory}, Geom. Funct. Anal. (2000), no.~Special Volume, Part II,
  560--673, {GAFA 2000 (Tel Aviv, 1999)}.

\bibitem[FvK12]{Frauenfelder_vKoert_Hormadner_index_of_symmetric_periodic_orbits}
U.~Frauenfelder and O.~van Koert, \emph{{The H\"ormander index of symmetric
  periodic orbits}}, 2012, arXiv:1208.4756.

\bibitem[MS98]{McDuff_Salamon_introduction_symplectic_topology}
D.~McDuff and D.~A. Salamon, \emph{Introduction to symplectic topology}, second
  ed., Oxford Mathematical Monographs, The Clarendon Press Oxford University
  Press, New York, 1998.

\bibitem[Won66]{Wonenburger_Transformations_which_are_products_of_two_involutions}
M.~J. Wonenburger, \emph{Transformations which are products of two
  involutions}, J. Math. Mech. \textbf{16} (1966), 327--338.

\end{thebibliography}
\end{document}